\documentclass[12pt]{amsart}
\usepackage{amsfonts,amsthm,amssymb,amsmath,cite}
\allowdisplaybreaks

\newtheorem{theorem}{Theorem}[section]

\theoremstyle{definition}
\newtheorem{definition}[theorem]{Definition}
\theoremstyle{remark}

\numberwithin{equation}{section}

\newcommand{\R}{{\mathbb{R}}}
\newcommand{\C}{{\mathbb{C}}}

\newcommand{\N}{{\mathbb{N}}}

\newcommand{\dt}{{\rm{d}}t}
\newcommand{\ddz}{\frac{\rm d}{{\rm d}z}}
\DeclareMathOperator{\im}{Im}

\begin{document}

\title[Matrix-valued discrete Dirac system]
{\textbf{Spectral properties of matrix-valued discrete Dirac system}}
\author[Y.~Aygar]{Yelda Aygar}
\address{University of Ankara,
         Faculty of Science,
         Department of Math\-e\-mat\-ics,
         06100, Ankara, Turkey}
\email{yaygar@science.ankara.edu.tr}
\author[E.~Bairamov]{Elgiz Bairamov}
\address{University of Ankara,
         Faculty of Science,
         Department of Math\-e\-mat\-ics,
         06100, Ankara, Turkey}
\email{bairamov@science.ankara.edu.tr}
\author[S.~Yardýmcý]{Seyhmus Yardimci}
\address{University of Ankara,
         Faculty of Science,
         Department of Math\-e\-mat\-ics,
         06100, Ankara, Turkey}
\email{yardimci@science.ankara.edu.tr}
\renewcommand{\subjclassname}{%
  \textup{2010} Mathematics Subject Classification}

\begin{abstract}
In this paper, we find a polynomial-type Jost solution of a self-adjoint
matrix-valued discrete Dirac system. Then we investigate analytical properties
and asymptotic behavior of this Jost solution. Using the Weyl compact
perturbation theorem, we prove that matrix-valued discrete Dirac system has
continuous spectrum filling the segment $[-2,2].$ Finally, we examine the properties of
the eigenvalues of this Dirac system and we prove that it has a finite number of simple real eigenvalues.
\end{abstract}
\maketitle


\section{\textbf{Introduction}}
Consider the boundary value problem (BVP) consisting of the
Sturm--Liouville equation
\begin{equation}\label{E1}
\begin{cases}-y''+q(x)y=\lambda^2y,& 0\leq x<\infty\medskip\\
y(0)=0,& \medskip\\
\end{cases}
\end{equation}
and the boundary condition%
where $q$ is a real-valued function and $\lambda$ is a spectral parameter. The bounded solution of \eqref{E1} satisfying the condition%
$$ \lim_{x\to\infty}y(x,\lambda)e^{-i\lambda x}=1,\quad
   \lambda\in\overline{\C}_{+}
   :=\left\{\lambda\in\C:\; \im\lambda\geq 0\right\}$$
will be denoted by $e(.,\lambda).$ The solution $e(.,\lambda)$ is called the Jost solution of \eqref{E1}. In \cite{Marchenko}, the author presented a condition depending on the function $q$ that guaranteed $e(.,\lambda)$ has an integral representation as
\begin{equation*}
   e(x,\lambda)=e^{i\lambda x}+\int_{x}^{\infty}K(x,t)e^{i\lambda t}\dt<\infty,
\end{equation*}
where the function $K$ is defined in terms of $q$. Moreover, the author showed that $e(.,\lambda)$ is analytic with respect to $\lambda$ in $%
\mathbb{C}
_{+}:=\left\{  \lambda\in%
\mathbb{C}
:\text{ }\operatorname{Im}\lambda>0\right\}$, continuous in $\overline{\C}_{+}$ and satisfies
$$e(x,\lambda)=e^{i\lambda x}[1+o(1)],\quad \lambda\in\overline{\C}_{+},\quad x\to\infty.$$ The function $e(\lambda):=e(0,\lambda)$ is
called Jost function of the BVP \eqref{E1}. The functions
$e(.,\lambda)$ and $e(\lambda)$ play an
important role in the solutions of direct and inverse problems of the quantum scattering
theory \cite{Marchenko,Chadan+Sabatier,Levitan,Novikov+Manakov}. The Jost solutions are especially useful in
the study of the spectral analysis of differential and difference operators \cite{Adývar+Bairamov,Bairamov+Celebi,Bairamov+Cakar+Krall,Bairamov*+Cakar*+Krall*,Bairamov+Karaman,Bairamov+Coskun,Krall+Bairamov+Cakar}. Therefore Jost solutions of Dirac systems, Schr\"{o}dinger and discrete Sturm--Liouville equations have been obtained in \cite{Gasymov+Levitan,Jaulent+Jean,Guseinov}. Discrete boundary value problems have been intensively studied in the last decade. The modeling of certain linear and nonlinear problems from economics, optimal control theory and other areas of study has led to the rapid development of the theory of difference equations. Also the spectral analysis of the difference equations has been treated by various authors in connection with the classical moment problem \cite{Akhiezer,Agarwal,Agarwal+Wong,Kelley+Peterson}. The spectral theory of the difference equations have also been applied to the solution of classes of nonlinear discrete Korteveg-de Vries equations and Toda lattices \cite{Toda}. Let us introduce the Hilbert space $\ell_{2}(\N,\C^{2m})$ consisting of all vector sequences $y=\{y_{n}\},$ $y_{n}=\binom{y_{n}^{\left(  1\right)  }}{y_{n}^{\left(  2\right)  }},$
where $y_{n}^{\left(  i\right)  }\in%
\mathbb{C}
^{m},$ $i=1,2,$ $n\in%
\mathbb{N}
$ and $%
\mathbb{C}
^{m}$ is $m-$dimensional $\left(  m<\infty\right)  $ Euclidean \ space. In $\ell_{2}(\N,\C^{2m})$, the norm and inner product are defined by
\[
\left\Vert y\right\Vert _{\ell_{2}}^{2}:=%
{\displaystyle\sum\limits_{n=1}^{\infty}}
\left(  \left\Vert y_{n}^{\left(  1\right)  }\right\Vert _{%
\mathbb{C}
^{m}}^{2}+\left\Vert y_{n}^{\left(  2\right)  }\right\Vert _{%
\mathbb{C}
^{m}}^{2}\right)  <\infty,
\]%
\[
\left\langle y,z\right\rangle _{\ell_{2}}:=%
{\displaystyle\sum\limits_{n=1}^{\infty}}
\left[  \left(  y_{n}^{\left(  1\right)  },z_{n}^{\left(  1\right)  }\right)
_{%
\mathbb{C}
^{m}}+\left(  y_{n}^{\left(  2\right)  },z_{n}^{\left(  2\right)  }\right)
_{_{%
\mathbb{C}
^{m}}}\right]  ,
\]
where $\left\Vert .\right\Vert _{%
\mathbb{C}
^{m}}$ and $\left(  .,.\right)  _{%
\mathbb{C}
^{m}}$ denote the norm and inner product in $%
\mathbb{C}
^{m},$ respectively. Now consider the matrix-valued discrete Dirac
system%
\begin{equation}\label{E2}
\begin{cases}
A_{n}y_{n+1}^{(2)}+B_{n}y_{n}^{(2}+P_{n}
y_{n}^{(1)}=\lambda y_{n}^{(1)}\\
A_{n-1}y_{n-1}^{(1)}+B_{n}y_{n}^{(1)}+Q_{n}
y_{n}^{(2)}=\lambda y_{n}^{(2)},
\end{cases}
\N=\left\{1,2,\ldots\right\},
\end{equation}
with the boundary condition
\begin{equation}\label{E3}
y_{0}^{(1)}=0,
\end{equation}
where $A_{n}$, $n\in N
\cup\left\{  0\right\} $ and $B_{n}$, $Q_{n},$ $P_{n},$ $n\in%
\mathbb{N}
$ are linear operators (matrices) acting in $%
\mathbb{C}
^{m}$. Throughout the paper, we will assume that $\det A_{n}\neq0$,
$A_{n}=A_{n}^{\ast}$ $(n\in%
\mathbb{N}
\cup\left\{ 0\}  \right)  $, $\det B_{n}\neq0$, $B_{n}=B_{n}^{\ast}$,
$Q_{n}=Q_{n}^{\ast}$ and $P_{n}=P_{n}^{\ast}$ $\left(  n\in%
\mathbb{N}
\right)  $, where $\ast$ denotes the adjoint operator. Let $L$ denote the
operator generated in $\ell_{2}\left(
\mathbb{N}
,%
\mathbb{C}
^{2m}\right)  $ by the BVP \eqref{E2}--\eqref{E3}. The
operator $L$ is self-adjoint, $i.e.$, $L=L^{\ast}$. In the following, we will
assume that the matrix sequences $\left\{  A_{n}\right\}  $, $\left\{
B_{n}\right\}  $, $\left\{  P_{n}\right\}  $, and $\left\{  Q_{n}\right\}$ $\left(  n\in%
\mathbb{N}
\right)$, satisfy
\begin{equation}\label{E4}
\sum_{n=1}^{\infty}
n\left(\|I-A_{n}\|+\|I+B_{n}\|+\|P_{n}\|+\| Q_{n}\|\right)<\infty,
\end{equation}
where $\|.\| $ and $I$ denote the matrix norm and identity matrix in $\mathbb{C}^{m}$, respectively.
The setup of this paper is as follows: In Section 2, we find a polynomial-type
Jost solution of \eqref{E2}, investigate analytical properties and asymptotic
behavior of this Jost solution. In Section 3, we show that $\sigma_{c}(L)=[-2,2]$, where $\sigma_{c}\left(  L\right)  $ denotes the
continuous spectrum of $L$. Also, we prove that under the condition \eqref{E4}, the operator $L$ has a finite number of simple real eigenvalues.
To the best of our knowledge, this paper is the first one that focuses on matrix-valued discrete Dirac system including a polynomial-type Jost solution.

\section{Jost solution of \eqref{E2}}
Assume $P_{n}\equiv Q_{n}\equiv0$, $B_{n}\equiv-I$ for all $n\in\N$, and $A_{n}\equiv I$ for all $n\in\N
\cup\left\{0\right\}$ in \eqref{E2}. Then we get
\begin{equation}\label{E5}
\left\{
\begin{array}
[c]{c}%
y_{n+1}^{\left(  2\right)  }-y_{n}^{\left(  2\right)  }=\left[  -iz-\left(
iz\right)  ^{-1}\right]  y_{n}^{\left(  1\right)  }\\
\\
y_{n-1}^{\left(  1\right)  }-y_{n}^{\left(  1\right)  }=\left[  -iz-\left(
iz\right)  ^{-1}\right]  y_{n}^{\left(  2\right)  }%
\end{array}
\right.
\end{equation}
for $\lambda=-iz-\left(  iz\right)^{-1}$.
It is clear that
\[
e_{n}\left(  z\right)  =\binom{e_{n}^{\left(  1\right)  }\left(  z\right)
}{e_{n}^{\left(  2\right)  }\left(  z\right)  }=\binom{z}{-i}z^{2n},\text{
\ \ \ \ \ }n\in%
\mathbb{N}
,
\]
is a solution of \eqref{E5}. Now, we will find the solution $\binom{F_{n}\left(
z\right)  }{^{G_{n}\left(  z\right)  }}$, $n\in%
\mathbb{N}
\cup\left\{  0\right\}  $ of \eqref{E2} for $\lambda=-iz-\left(  iz\right)
^{-1}$, satisfying the condition%
\begin{equation}\label{E6}
\binom{F_{n}\left(  z\right)  }{^{G_{n}\left(  z\right)  }}=\left[  I+o\left(
1\right)  \right]  e_{n}\left(  z\right), \text{ \ \ \ }\left\vert
z\right\vert =1,\text{ }n\rightarrow\infty.
\end{equation}
The solution $\binom{F_{n}\left(  z\right)  }{^{G_{n}\left(  z\right)  }}$,
$n\in \N$ is called the Jost solution of \eqref{E2} for $\lambda=-iz-\left(iz\right)^{-1}$.
\begin{theorem}\label{T1}
Assume \eqref{E4}. Then for $\lambda=-iz-\left(iz\right)^{-1}$ and $\left\vert
z\right\vert =1$, \eqref{E2} has the solution
$\binom{F_{n}\left(  z\right)  }{^{G_{n}\left(  z\right)  }}$, $n\in\N$ having the representation
\begin{equation}\label{E7}
\binom{F_{n}\left(  z\right)  }{^{G_{n}\left(  z\right)  }}=T_{n}\left(  I+%
{\displaystyle\sum\limits_{m=1}^{\infty}}
K_{nm}z^{2m}\right)  \left(
\begin{array}
[c]{c}%
z\\
-i
\end{array}
\right)  z^{2n},\text{ }n\in
\mathbb{N},
\end{equation}
\begin{equation}\label{E8}
F_{0}\left(  z\right)  =T_{0}^{11}z+T_{0}^{11}%
{\displaystyle\sum\limits_{m=1}^{\infty}}
K_{0m}^{11}z^{2m+1}-iT_{0}^{11}%
{\displaystyle\sum\limits_{m=1}^{\infty}}
K_{0m}^{12}z^{2m},
\end{equation}
where $K_{nm}=\left(
\begin{array}
[c]{cc}%
K_{nm}^{11} & K_{nm}^{12}\\
K_{nm}^{21} & K_{nm}^{22}%
\end{array}
\right)$, $T_{n}=\left(
\begin{array}
[c]{cc}%
T_{n}^{11} & T_{n}^{12}\\
T_{n}^{21} & T_{n}^{22}%
\end{array}
\right)$ and these are expressed in terms of $\left\{  A_{n}\right\}$, $\left\{B_{n}\right\}$, $\left\{P_{n}\right\}$ and $\left\{Q_{n}\right\}$.
\begin{proof}
Substituting $\binom{F_{n}\left(z\right)}{^{G_{n}\left(  z\right)}}$
defined by \eqref{E7} into \eqref{E2} and taking $\lambda=-iz-\left(iz\right)^{-1}$,
$|z| =1$, we get
\begin{align*}
T_{n}^{12}  & =0\text{, \ }T_{n}^{22}=\left(
{\displaystyle\prod\limits_{p=n}^{\infty}}
\left(  -1\right)  ^{n-p}A_{p}B_{p}\right)  ^{-1},\\
T_{n}^{11}  & =-B_{n}\left(
{\displaystyle\prod\limits_{p=n}^{\infty}}
\left(  -1\right)  ^{n-p}A_{p}B_{p}\right)  ^{-1},\\
T_{n}^{21}  & =-Q_{n}T_{n}^{22}-A_{n-1}T_{n-1}^{11}\left(
{\displaystyle\sum\limits_{p=n}^{\infty}}
\left(  T_{p}^{11}\right)  ^{-1}B_{p}Q_{p}T_{p}^{22}-%
{\displaystyle\sum\limits_{p=n}^{\infty}}
\left(  T_{p}^{11}\right)  ^{-1}P_{p}T_{p}^{11}\right),
\end{align*}%
\begin{align*}
K_{n1}^{12}  & =%
{\displaystyle\sum\limits_{p=n+1}^{\infty}}
\left(  T_{p}^{11}\right)  ^{-1}B_{p}Q_{p}T_{p}^{22}-%
{\displaystyle\sum\limits_{p=n+1}^{\infty}}
\left(  T_{p}^{11}\right)  ^{-1}P_{p}T_{p}^{11},\\
K_{n1}^{11}  & =%
{\displaystyle\sum\limits_{p=n+1}^{\infty}}
\left[  -I+\left(  T_{p}^{11}\right)  ^{-1}\left(  B_{p}^{2}T_{p}^{11}%
+A_{p}T_{p+1}^{22}+B_{p}Q_{p}T_{p}^{21}+B_{p}T_{p}^{22}+P_{p}T_{p}^{11}%
K_{p1}^{12}\right)  \right]  ,\\
K_{n1}^{22}  & =\left(  T_{n}^{22}\right)  ^{-1}\left[  B_{n}T_{n}^{11}%
+Q_{n}T_{n}^{21}+T_{n}^{22}+A_{n-1}T_{n-1}^{11}K_{n-1,1}^{11}-T_{n}^{21}%
K_{n1}^{12}\right]  ,\\
K_{n1}^{21}  & =%
{\displaystyle\sum\limits_{p=n+1}^{\infty}}
\left(  T_{p}^{22}\right)  ^{-1}T_{p}^{21}\left(  K_{p1}^{11}-I\right)  +%
{\displaystyle\sum\limits_{p=n+1}^{\infty}}
\left[  \left(  T_{p}^{22}\right)  ^{-1}\left(  B_{p}T_{p}^{11}+Q_{p}%
T_{p}^{21}\right)  K_{p1}^{12}\right] \\
& +%
{\displaystyle\sum\limits_{p=n+1}^{\infty}}
\left(  T_{p}^{22}\right)  ^{-1}Q_{p}T_{p}^{22}K_{p1}^{22}+%
{\displaystyle\sum\limits_{p=n+1}^{\infty}}
\left(  T_{p}^{22}\right)  ^{-1}A_{p-1}T_{p-1}^{11}K_{p-1,1}^{12}\\
& +%
{\displaystyle\sum\limits_{p=n+1}^{\infty}}
\left(  T_{p}^{22}\right)  ^{-1}A_{p-1}^{2}T_{p}^{21}+%
{\displaystyle\sum\limits_{p=n}^{\infty}}
\left(  T_{p+1}^{22}\right)  ^{-1}\left[  A_{p}P_{p}T_{p}^{11}+A_{p}B_{p}%
T_{p}^{21}\right]  K_{p1}^{11},
\end{align*}%
\begin{align*}
K_{n2}^{12}  & =%
{\displaystyle\sum\limits_{p=n+1}^{\infty}}
\left[  \left(  T_{p}^{11}\right)  ^{-1}B_{p}Q_{p}\left(  T_{p}^{21}%
K_{p1}^{12}+T_{p}^{22}K_{p1}^{22}\right)  \right]  -%
{\displaystyle\sum\limits_{p=n+1}^{\infty}}
K_{p1}^{12}\\
& +%
{\displaystyle\sum\limits_{p=n+1}^{\infty}}
\left(  T_{p}^{11}\right)  ^{-1}\left[  B_{p}^{2}T_{p}^{11}K_{p1}^{12}%
-P_{p}T_{p}^{11}K_{p1}^{11}-A_{p}T_{p+1}^{21}-B_{p}T_{p}^{21}\right],\\
K_{n2}^{11}  & =%
{\displaystyle\sum\limits_{p=n+1}^{\infty}}
\left(  T_{p}^{11}\right)  ^{-1} \left[  B_{p}^{2}T_{p}^{11}K_{p1}^{11}%
+P_{p}T_{p}^{11}K_{p2}^{12}+B_{p}T_{p}^{22}K_{p1}^{22}+B_{p}T_{p}^{21}%
K_{p1}^{12}\right] \\
& +%
{\displaystyle\sum\limits_{p=n+1}^{\infty}}
\left(  T_{p}^{11}\right)  ^{-1}\left[  B_{p}Q_{p}T_{p}^{21}K_{p1}^{11}%
+B_{p}Q_{p}T_{p}^{22}K_{p1}^{21}\right]  -%
{\displaystyle\sum\limits_{p=n+1}^{\infty}}
K_{p1}^{11}\\
& +%
{\displaystyle\sum\limits_{p=n+2}^{\infty}}
\left(  T_{p-1}^{11}\right)  ^{-1}\left[  A_{p-1}T_{p}^{22}K_{p1}^{22}%
+A_{p-1}T_{p}^{21}K_{p1}^{12}\right],\\
&\\K_{n2}^{22}  & =-%
{\displaystyle\sum\limits_{p=n+1}^{\infty}}
\left(  T_{p}^{22}\right)  ^{-1} \left[  B_{p}T_{p}^{11}K_{p1}^{11}-T_{p}%
^{21}K_{p2}^{12}+Q_{p}T_{p}^{21}K_{p1}^{11}+Q_{p}T_{p}^{22}K_{p1}^{21}%
+T_{p}^{21}K_{p1}^{12}\right] \\
& +%
{\displaystyle\sum\limits_{p=n+1}^{\infty}}
\left(  T_{p}^{22}\right)  ^{-1}\left[  A_{p-1}P_{p-1}T_{p-1}^{11}%
K_{p-1,2}^{12}+A_{p-1}B_{p-1}T_{p-1}^{21}K_{p-1,2}^{12}\right] \\
& +%
{\displaystyle\sum\limits_{p=n+1}^{\infty}}
\left(  T_{p}^{22}\right)  ^{-1}\left[  -A_{p-1}T_{p-1}^{11}K_{p-1,1}%
^{11}\right]  -%
{\displaystyle\sum\limits_{p=n+1}^{\infty}}
K_{p1}^{22}\\
&
{\displaystyle\sum\limits_{p=n+1}^{\infty}}
\left(  T_{p}^{22}\right)  ^{-1}\left[  A_{p-1}^{2}T_{p}^{22}K_{p1}%
^{12}+A_{p-1}^{2}T_{p}^{21}K_{p,1}^{12}\right]  ,\\
K_{n2}^{21}  & =%
{\displaystyle\sum\limits_{p=n}^{\infty}}
\left(  T_{p+1}^{22}\right)  ^{-1} \left[  A_{p}T_{p}^{11}K_{p2}^{12}%
+A_{p}P_{p}T_{p}^{11}K_{p2}^{11}+A_{p}B_{p}T_{p}^{21}K_{p2}^{11}\right] \\
& +%
{\displaystyle\sum\limits_{p=n+1}^{\infty}}
\left(  T_{p}^{22}\right)  ^{-1}\left[  A_{p-1}^{2}T_{p}^{21}K_{p1}%
^{11}+A_{p-1}^{2}T_{p}^{22}K_{p,1}^{21}+B_{p}T_{p}^{11}K_{p2}^{12}\right]  -%
{\displaystyle\sum\limits_{p=n+1}^{\infty}}
K_{p1}^{21}\\
& +%
{\displaystyle\sum\limits_{p=n+1}^{\infty}}
\left(  T_{p}^{22}\right)  ^{-1}\left[  Q_{p}T_{p}^{21}K_{p2}^{12}+Q_{p}%
T_{p}^{22}K_{p2}^{22}+T_{p}^{21}K_{p2}^{11}-T_{p}^{21}K_{p1}^{11}\right],
\end{align*}%
where $n\in \N$. Furthermore, for $m\geq3$ and $n\in\N$, we obtain that
\begin{align*}
K_{nm}^{12}  & =-%
{\displaystyle\sum\limits_{p=n+1}^{\infty}}
\left(  T_{p}^{11}\right)  ^{-1} \left[  P_{p}T_{p}^{11}K_{p,m-1}^{11}%
-B_{p}^{2}T_{p}^{11}K_{p,m-1}^{12}+B_{p}T_{p}^{21}K_{p,m-2}^{11}\right] \\
& +%
{\displaystyle\sum\limits_{p=n+1}^{\infty}}
\left(  K_{p,m-2}^{21}-K_{p,m-1}^{12}\right)  +%
{\displaystyle\sum\limits_{p=n+1}^{\infty}}
\left(  T_{p}^{11}\right)  ^{-1}\left[  B_{p}Q_{p}^{2}T_{p}^{21}K_{p,m-2}%
^{11}\right] \\
& -%
{\displaystyle\sum\limits_{p=n+1}^{\infty}}
\left(  T_{p}^{11}\right)  ^{-1}\left[  A_{p}T_{p+1}^{21}K_{p+1,m-2}%
^{11}+A_{p}T_{p+1}^{22}K_{p+1,m-2}^{21}\right] \\
& +%
{\displaystyle\sum\limits_{p=n+1}^{\infty}}
\left(  T_{p}^{11}\right)  ^{-1}\left[  B_{p}Q_{p}A_{p-1}T_{p-1}%
^{11}K_{p-1,m-1}^{11}+B_{p}Q_{p}T_{p}^{11}K_{p,m-2}^{11}\right] \\
& +%
{\displaystyle\sum\limits_{p=n+1}^{\infty}}
\left(  T_{p}^{11}\right)  ^{-1}\left[  B_{p}Q_{p}^{2}T_{p}^{22}K_{p,m-2}%
^{21}+B_{p}Q_{p}T_{p}^{21}K_{p,m-2}^{12}+B_{p}Q_{p}T_{p}^{22}K_{p,m-2}%
^{22}\right]  ,
\end{align*}%
\begin{align*}
K_{nm}^{11}  & =+%
{\displaystyle\sum\limits_{p=n+1}^{\infty}}
\left(  T_{p}^{11}\right)  ^{-1} \left[  P_{p}T_{p}^{11}K_{pm}^{12}+B_{p}%
^{2}T_{p}^{11}K_{p,m-1}^{11}+B_{p}Q_{p}T_{p}^{21}K_{p,m-1}^{11}\right] \\
& -%
{\displaystyle\sum\limits_{p=n+1}^{\infty}}
K_{p,m-1}^{11}+%
{\displaystyle\sum\limits_{p=n+1}^{\infty}}
\left(  T_{p}^{11}\right)  ^{-1}\left[  B_{p}Q_{p}T_{p}^{22}K_{p,m-1}%
^{21}\right] \\
& +%
{\displaystyle\sum\limits_{p=n+1}^{\infty}}
\left(  T_{p}^{11}\right)  ^{-1}\left[  A_{p}T_{p+1}^{22}K_{p+1,m-1}%
^{22}+A_{p}T_{p+1}^{21}K_{p+1,m-1}^{12}\right] \\
& +%
{\displaystyle\sum\limits_{p=n+1}^{\infty}}
\left(  T_{p}^{11}\right)  ^{-1}\left[  B_{p}T_{p}^{21}K_{p,m-1}^{12}%
+B_{p}T_{p}^{22}K_{p,m-1}^{22}\right],
\end{align*}%
\begin{align*}
K_{nm}^{22}  & =K_{n,m-1}^{22}+\left(  T_{n}^{22}\right)  ^{-1}\left[
A_{n-1}T_{n-1}^{11}K_{n-1,m}^{11}-T_{n}^{21}K_{nm}^{21}+B_{n}T_{n}%
^{11}K_{n,m-1}^{11}\right] \\
& +\left(  T_{n}^{22}\right)  ^{-1}\left[  Q_{n}T_{n}^{21}K_{n,m-1}^{11}%
+Q_{n}T_{n}^{22}K_{n,m-1}^{21}+T_{n}^{21}K_{n,m-1}^{12}\right],
\end{align*}%
\begin{align*}
K_{nm}^{21}  & =K_{n,m-1}^{21}+\left(  T_{n}^{22}\right)  ^{-1}\left[
-A_{n-1}T_{n-1}^{11}K_{n-1,m+1}^{12}-B_{n}T_{n}^{11}K_{nm}^{12}\right] \\
& +\left(  T_{n}^{22}\right)  ^{-1}\left[  -Q_{n}T_{n}^{21}K_{nm}^{12}%
+T_{n}^{21}K_{n,m-1}^{11}-Q_{n}T_{n}^{22}K_{nm}^{22}-T_{n}^{21}K_{nm}%
^{11}\right].
\end{align*}
By the condition \eqref{E4}, the infinite products and the series in the definition
of $T_{n}^{ij}$ and $K_{nm}^{ij}$ $\left(i,j=1,2\right)$ are absolutely
convergent. Therefore, $T_{n}^{ij}$ and $K_{nm}^{ij}$ $\left(
i,j=1,2\right) $ can uniquely be defined by $\{A_{n}\}$, $n\in\N
\cup\left\{0\right\}$, $\left\{B_{n}\right\}$, $\left\{
P_{n}\right\}$ and $\left\{  Q_{n}\right\}$, $n\in\N$, $i.e.$, the system \eqref{E2} for $\lambda=-iz-\left(  iz\right)  ^{-1}$ has the
solution $\binom{F_{n}\left(  z\right)  }{^{G_{n}\left(  z\right)  }}$ given
by \eqref{E7}.
\end{proof}
\end{theorem}
\begin{theorem}\label{T2}
If the condition \eqref{E4} holds, then
\begin{equation}\label{E9}
\left\Vert K_{nm}^{ij}\right\Vert \leq C%
{\displaystyle\sum\limits_{p=n+\left\lfloor \frac{m}{2}\right\rfloor }%
^{\infty}}
\left(  \left\Vert I-A_{p}\right\Vert +\left\Vert I+B_{p}\right\Vert
+\left\Vert Q_{p}\right\Vert +\left\Vert P_{p}\right\Vert \right),\quad i,j=1,2,
\end{equation}
where $\left\lfloor \frac{m}{2}\right\rfloor $ is the integer part of
$\frac{m}{2}$ and $C>0$ is a constant.
\end{theorem}
\begin{proof}
We will use the method of induction to prove the theorem. For $m=1,$ we get that
\begin{align*}
\left\Vert K_{n1}^{12}\right\Vert  & =\left\Vert
{\displaystyle\sum\limits_{p=n+1}^{\infty}}
\left(  T_{p}^{11}\right)  ^{-1}\left[  B_{p}Q_{p}T_{p}^{22}-P_{p}T_{p}%
^{11}\right]  \right\Vert \leq A%
{\displaystyle\sum\limits_{p=n+1}^{\infty}}
\left\Vert B_{p}Q_{p}T_{p}^{22}-P_{p}T_{p}^{11}\right\Vert \\
& \leq A^{^{\prime}}%
{\displaystyle\sum\limits_{p=n+1}^{\infty}}
\left\Vert Q_{p}\right\Vert +%
{\displaystyle\sum\limits_{p=n+1}^{\infty}}
\left\Vert P_{p}\right\Vert \leq C%
{\displaystyle\sum\limits_{p=n+1}^{\infty}}
\left(  \left\Vert Q_{p}\right\Vert +\left\Vert P_{p}\right\Vert \right) \\
& \leq C%
{\displaystyle\sum\limits_{p=n}^{\infty}}
\left(  \left\Vert I-A_{p}\right\Vert +\left\Vert I+B_{p}\right\Vert
+\left\Vert Q_{p}\right\Vert +\left\Vert P_{p}\right\Vert \right) \\
& =C%
{\displaystyle\sum\limits_{p=n+\left\lfloor \frac{1}{2}\right\rfloor }%
^{\infty}}
\left(  \left\Vert I-A_{p}\right\Vert +\left\Vert I+B_{p}\right\Vert
+\left\Vert Q_{p}\right\Vert +\left\Vert P_{p}\right\Vert \right),
\end{align*}
where $A=\left\Vert \left(  T_{p}^{11}\right)  ^{-1}\right\Vert$,
$A^{^{\prime}}=A\left\Vert B_{p}\right\Vert \left\Vert T_{p}^{22}\right\Vert$, $C=\max\left\{  1,A^{^{\prime}}\right\}$. Similar to this inequality, we can get \eqref{E9} for $K_{n1}^{11}$, $K_{n1}^{22}$ and
$K_{n1}^{21}$. Now, if we suppose that \eqref{E9} is correct for $m=k$, then we can write
\begin{align*}
\left\Vert K_{n,k+1}^{12}\right\Vert  & \leq\left\Vert
{\displaystyle\sum\limits_{p=n+1}^{\infty}}
K_{p,k-1}^{21}-%
{\displaystyle\sum\limits_{p=n+1}^{\infty}}
K_{p,k}^{12}\right\Vert \\
& +\left\Vert
{\displaystyle\sum\limits_{p=n+1}^{\infty}}
\left(  T_{p}^{11}\right)  ^{-1}\left\{  -P_{p}T_{p}^{11}K_{pk}^{11}+B_{p}%
^{2}T_{p}^{11}K_{pk}^{12}-B_{p}T_{p}^{21}K_{p,k-1}^{11}\right\}  \right\Vert
\\
& +\left\Vert
{\displaystyle\sum\limits_{p=n+1}^{\infty}}
\left(  T_{p}^{11}\right)  ^{-1}\left\{  -A_{p}T_{p+1}^{21}K_{p+1,k-1}%
^{11}-A_{p}T_{p+1}^{22}K_{p+1,k-1}^{21}\right\}  \right\Vert \\
& +\left\Vert
{\displaystyle\sum\limits_{p=n+1}^{\infty}}
\left(  T_{p}^{11}\right)  ^{-1}\left\{  B_{p}Q_{p}A_{p-1}T_{p-1}%
^{11}K_{p-1,k}^{11}+B_{p}Q_{p}B_{p}T_{p}^{11}K_{p,k-1}^{11}\right\}
\right\Vert \\
& +\left\Vert
{\displaystyle\sum\limits_{p=n+1}^{\infty}}
\left(  T_{p}^{11}\right)  ^{-1}\left\{  B_{p}Q_{p}^{2}T_{p}^{21}%
K_{p,k-1}^{11}+B_{p}Q_{p}T_{p}^{22}K_{p,k-1}^{21}\right\}  \right\Vert \\
& +\left\Vert
{\displaystyle\sum\limits_{p=n+1}^{\infty}}
\left(  T_{p}^{11}\right)  ^{-1}\left\{  B_{p}Q_{p}T_{p}^{21}K_{p,k-1}%
^{12}+B_{p}Q_{p}T_{p}^{22}K_{p,k-1}^{22}\right\}  \right\Vert.
\end{align*}
If we use $T_{p+1}^{22}=A_{p}T_{p}^{11}$ for last inequality, we find
\begin{align*}
\left\Vert K_{n,k+1}^{12}\right\Vert  & \leq%
{\displaystyle\sum\limits_{p=n+1}^{\infty}}
\left\Vert K_{p,k-1}^{21}-K_{p+1,k-1}^{21}\right\Vert +%
{\displaystyle\sum\limits_{p=n+1}^{\infty}}
\left\Vert \left(  T_{p}^{11}\right)  ^{-1}\left(  I-A_{p}^{2}\right)
T_{p}^{11}K_{p+1,k-1}^{21}\right\Vert \\
& +%
{\displaystyle\sum\limits_{p=n+1}^{\infty}}
\left\Vert \left(  T_{p}^{11}\right)  ^{-1}\left(  -B_{p}-I\right)  T_{p}%
^{21}K_{p,k-1}^{11}\right\Vert \\
& +%
{\displaystyle\sum\limits_{p=n+1}^{\infty}}
\left\Vert \left(  T_{p}^{11}\right)  ^{-1}\left(  I-A_{p}\right)
T_{p+1}^{21}K_{p+1,k-1}^{21}\right\Vert \\
& +%
{\displaystyle\sum\limits_{p=n+1}^{\infty}}
\left\Vert \left(  T_{p}^{11}\right)  ^{-1}\left(  T_{p}^{21}K_{p,k-1}%
^{11}-T_{p+1}^{21}K_{p+1,k-1}^{11}\right)  \right\Vert \\
& +C^{^{\prime\prime}}%
{\displaystyle\sum\limits_{p=n+1}^{\infty}}
\left\Vert B_{p}+I\right\Vert
{\displaystyle\sum\limits_{s=p+\left\lfloor \frac{k}{2}\right\rfloor }%
^{\infty}}
\left\Vert N_{s}\right\Vert +%
{\displaystyle\sum\limits_{p=n+1}^{\infty}}
\left\Vert P_{p}\right\Vert
{\displaystyle\sum\limits_{s=p+\left\lfloor \frac{k}{2}\right\rfloor }%
^{\infty}}
\left\Vert N_{s}\right\Vert \\
& +B%
{\displaystyle\sum\limits_{p=n+1}^{\infty}}
\left\Vert Q_{p}\right\Vert
{\displaystyle\sum\limits_{s=p+\left\lfloor \frac{k-2}{2}\right\rfloor
}^{\infty}}
\left\Vert N_{s}\right\Vert +D%
{\displaystyle\sum\limits_{p=n+1}^{\infty}}
\left\Vert Q_{p}\right\Vert
{\displaystyle\sum\limits_{s=p+\left\lfloor \frac{k-1}{2}\right\rfloor
}^{\infty}}
\left\Vert N_{s}\right\Vert,
\end{align*}
where $\left\Vert N_{s}\right\Vert =\left\Vert I-A_{s}\right\Vert +\left\Vert
I+B_{s}\right\Vert +\left\Vert Q_{s}\right\Vert +\left\Vert P_{s}\right\Vert$ and $C^{^{\prime\prime}}$, $B$, $D$ are constants. It follows from that
\begin{align*}
\left\Vert K_{n,k+1}^{12}\right\Vert & \leq\left\Vert K_{n+1,k-1}^{21}\right\Vert +D^{^{\prime}}%
{\displaystyle\sum\limits_{p=n+1}^{\infty}}
\left\Vert I-A_{p}\right\Vert
{\displaystyle\sum\limits_{s=p+1+\left\lfloor \frac{k}{2}\right\rfloor
}^{\infty}}
\left\Vert N_{s}\right\Vert \\
& +D^{^{\prime\prime}}%
{\displaystyle\sum\limits_{p=n+1}^{\infty}}
\left\Vert B_{p}+I\right\Vert
{\displaystyle\sum\limits_{s=p+\left\lfloor \frac{k-1}{2}\right\rfloor
}^{\infty}}
\left\Vert N_{s}\right\Vert \\
& +D^{^{\prime\prime\prime}}%
{\displaystyle\sum\limits_{p=n+1}^{\infty}}
\left\Vert I-A_{p}\right\Vert
{\displaystyle\sum\limits_{s=p+1+\left\lfloor \frac{k-1}{2}\right\rfloor
}^{\infty}}
\left\Vert N_{s}\right\Vert +T\left\Vert K_{n+1,k-1}^{11}\right\Vert \\
& +C^{\prime\prime}%
{\displaystyle\sum\limits_{p=n+1}^{\infty}}
\left\Vert B_{p}+I\right\Vert
{\displaystyle\sum\limits_{s=p+\left\lfloor \frac{k}{2}\right\rfloor }%
^{\infty}}
\left\Vert N_{s}\right\Vert
& \\
& +
\left\Vert P_{p}\right\Vert
{\displaystyle\sum\limits_{s=p+\left\lfloor \frac{k}{2}\right\rfloor }%
^{\infty}}
\left\Vert N_{s}\right\Vert +B%
{\displaystyle\sum\limits_{p=n+1}^{\infty}}
\left\Vert Q_{p}\right\Vert
{\displaystyle\sum\limits_{s=p+\left\lfloor \frac{k-2}{2}\right\rfloor
}^{\infty}}
\left\Vert N_{s}\right\Vert \\
& +D%
{\displaystyle\sum\limits_{p=n+1}^{\infty}}
\left\Vert Q_{p}\right\Vert
{\displaystyle\sum\limits_{s=p+\left\lfloor \frac{k-1}{2}\right\rfloor
}^{\infty}}
\left\Vert N_{s}\right\Vert,
\end{align*}
where $D^{^{\prime}}$, $D^{^{\prime\prime}}$, $D^{^{\prime\prime\prime}}$ and $T$ are also constants. Using last inequality, we obtain
\begin{align*}
\left\Vert K_{n,k+1}^{12}\right\Vert  & \leq C%
{\displaystyle\sum\limits_{p=n+1+\left\lfloor \frac{k-1}{2}\right\rfloor
}^{\infty}}
\left\Vert N_{p}\right\Vert +TC%
{\displaystyle\sum\limits_{p=n+1+\left\lfloor \frac{k-1}{2}\right\rfloor
}^{\infty}}
\left\Vert N_{p}\right\Vert \\
& +\left(  D^{^{\prime}}+D^{^{\prime\prime\prime}}\right)
{\displaystyle\sum\limits_{p=n+1}^{\infty}}
\left\Vert I-A_{p}\right\Vert
{\displaystyle\sum\limits_{s=p+\left\lfloor \frac{k+1}{2}\right\rfloor
}^{\infty}}
\left\Vert N_{s}\right\Vert \\
& +\max\left\{  D^{^{\prime\prime}},D\right\}
{\displaystyle\sum\limits_{p=n+1}^{\infty}}
\left(  \left\Vert B_{p}+I\right\Vert +\left\Vert Q_{p}\right\Vert \right)
{\displaystyle\sum\limits_{s=p+\left\lfloor \frac{k-1}{2}\right\rfloor
}^{\infty}}
\left\Vert N_{s}\right\Vert \\
& +\max\left\{  C^{^{\prime\prime}},1\right\}
{\displaystyle\sum\limits_{p=n+1}^{\infty}}
\left(  \left\Vert B_{p}+I\right\Vert +\left\Vert P_{p}\right\Vert \right)
{\displaystyle\sum\limits_{s=p+\left\lfloor \frac{k}{2}\right\rfloor }%
^{\infty}}
\left\Vert N_{s}\right\Vert \\
& +B%
{\displaystyle\sum\limits_{p=n+1}^{\infty}}
\left\Vert N_{p}\right\Vert
{\displaystyle\sum\limits_{s=p+\left\lfloor \frac{k-2}{2}\right\rfloor
}^{\infty}}
\left\Vert N_{s}\right\Vert
\end{align*}
and
\begin{align*}
\left\Vert K_{n,k+1}^{12}\right\Vert & \leq Z%
{\displaystyle\sum\limits_{p=n+\left\lfloor \frac{k+1}{2}\right\rfloor
}^{\infty}}
\left\Vert N_{p}\right\Vert +Y%
{\displaystyle\sum\limits_{p=n+1}^{\infty}}
\left\Vert N_{p}\right\Vert
{\displaystyle\sum\limits_{s=p+\left\lfloor \frac{k-2}{2}\right\rfloor
}^{\infty}}
\left\Vert N_{s}\right\Vert \\
& \leq Z%
{\displaystyle\sum\limits_{p=n+\left\lfloor \frac{k+1}{2}\right\rfloor
}^{\infty}}
\left\Vert N_{p}\right\Vert +Y\left\{
{\displaystyle\sum\limits_{p=n+1}^{\infty}}
\left\Vert N_{p}\right\Vert
{\displaystyle\sum\limits_{s=p+\left\lfloor \frac{k}{2}\right\rfloor }%
^{\infty}}
\left\Vert N_{s}\right\Vert \right\}
& \\
& \leq Z%
{\displaystyle\sum\limits_{p=n+\left\lfloor \frac{k+1}{2}\right\rfloor
}^{\infty}}
\left\Vert N_{p}\right\Vert +Y^{^{\prime}}%
{\displaystyle\sum\limits_{p=n+\left\lfloor \frac{k}{2}\right\rfloor }%
^{\infty}}
\left\Vert N_{p}\right\Vert \\
& \leq2Z%
{\displaystyle\sum\limits_{p=n+\left\lfloor \frac{k+1}{2}\right\rfloor
}^{\infty}}
\left\Vert N_{p}\right\Vert +Y^{^{\prime}}%
{\displaystyle\sum\limits_{p=n+\left\lfloor \frac{k+1}{2}\right\rfloor
}^{\infty}}
\left\Vert N_{p}\right\Vert \leq G%
{\displaystyle\sum\limits_{p=n+\left\lfloor \frac{k+1}{2}\right\rfloor
}^{\infty}}
\left\Vert N_{p}\right\Vert ,
\end{align*}
where $C+TC=Z$, $Y=D^{^{\prime}}+D^{^{\prime\prime\prime}}+B+\max\left\{
C^{^{\prime\prime}},1\right\}  +\max\left\{  D^{^{\prime\prime}},D\right\}$,
$Y^{^{\prime}}=Y
{\displaystyle\sum\limits_{p=n+1}^{\infty}}
\left\Vert N_{p}\right\Vert$ and $2Z+Y^{^{\prime}}=G$. Similar to
$K_{n,k+1}^{12}$, we can easily obtain \eqref{E9} for $K_{n,k+1}^{11}$,
$K_{n,k+1}^{21}$ and $K_{n,k+1}^{22}.$
\end{proof}
It follows from \eqref{E7} and \eqref{E9} that
$\binom{F_{n}\left(  z\right)  }{^{G_{n}\left(  z\right)  }}$ $n\in%
\mathbb{N}
\cup\left\{  0\right\}$ has analytic continuation from $D_{0}:=\left\{z\in\C:|z|=1\right\}$ to $\left\{z\in\C: |z| <1\right\}\backslash\left\{0\right\}$.
\begin{theorem}\label{T3}
Assume \eqref{E4}. Then the Jost solution satisfies
\begin{equation}\label{E10}
\binom{F_{n}\left(  z\right)  }{^{G_{n}\left(  z\right)  }}=\left[  I+o\left(
1\right)  \right]  \binom{z}{-i}z^{2n},\text{
}n\to\infty
\end{equation}
for $\text{ \ \ }z\in D:=\left\{  z\in%
\mathbb{C}
:\left\vert z\right\vert \leq1\right\}  \backslash\left\{  0\right\}$.
\end{theorem}
\begin{proof}
It follows from \eqref{E7} that
\begin{equation*}
\binom{F_{n}\left(  z\right)  }{^{G_{n}\left(  z\right)  }}=\left(
\begin{array}
[c]{cc}%
T_{n}^{11} & T_{n}^{12}\\
T_{n}^{21} & T_{n}^{22}%
\end{array}
\right)  \left[  \left(
\begin{array}
[c]{cc}%
1 & 0\\
0 & 1
\end{array}
\right)  +%
{\displaystyle\sum\limits_{m=1}^{\infty}}
\left(
\begin{array}
[c]{cc}%
K_{nm}^{11} & K_{nm}^{12}\\
K_{nm}^{21} & K_{nm}^{22}%
\end{array}
\right)  z^{2m}\right]  \binom{z^{2n+1}}{-iz^{2n}},
\end{equation*}
then using \eqref{E4}, \eqref{E9} and the definition of $T_{n}^{ij}$ for $i,j=1,2$, we get
\begin{equation}\label{E11}
\left(
\begin{array}
[c]{cc}%
T_{n}^{11} & T_{n}^{12}\\
T_{n}^{21} & T_{n}^{22}%
\end{array}
\right)  \to I,\text{ }n\to \infty,
\end{equation}
and
\begin{equation}\label{E12}
{\displaystyle\sum\limits_{m=1}^{\infty}}
\left(
\begin{array}
[c]{cc}%
K_{nm}^{11} & K_{nm}^{12}\\
K_{nm}^{21} & K_{nm}^{22}%
\end{array}
\right)  z^{2m}=o\left(  1\right),\text{ }z\in D,\text{ }n\to
\infty.
\end{equation}
From \eqref{E7}, \eqref{E11} and \eqref{E12}, we find \eqref{E10}.
\end{proof}
\section{\textbf{Continuous and discrete spectrum of }$L$}
\begin{theorem}\label{T4}
Under the condition \eqref{E4}, $\sigma_{c}\left(  L\right)  =\left[-2,2\right]$.
\end{theorem}
\begin{proof}
Let $L_{0}$ denote the operator generated in $\ell_{2}(\N,\C^{2m})$ by the difference expression
\[
\left(l_{1}y\right)  _{n}:=\left\{
\begin{array}
[c]{c}%
y_{n+1}^{\left(  2\right)  }-y_{n}^{\left(  2\right)  }\\
y_{n-1}^{\left(  1\right)  }-y_{n}^{\left(  1\right)  }%
\end{array}
\right.
\]
with the boundary condition $y_{0}^{\left(  1\right)  }=0$. We also define the operator $J$ in $\ell_{2}(\N,\C^{2m})$ by
\begin{align*}
J\binom{y_{n}^{\left(  1\right)  }}{y_{n}^{\left(  2\right)  }}  & :=\left(
\begin{array}
[c]{cc}%
P_{n} & 0\\
0 & Q_{n}%
\end{array}
\right)  \binom{y_{n}^{\left(  1\right)  }}{y_{n}^{\left(  2\right)  }%
}+\left(
\begin{array}
[c]{cc}%
I+B_{n} & 0\\
0 & I+B_{n}%
\end{array}
\right)  \binom{y_{n}^{\left(  2\right)  }}{y_{n}^{\left(  1\right)  }}\\
& +\left(
\begin{array}
[c]{cc}%
A_{n}-I & 0\\
0 & A_{n-1}-I
\end{array}
\right)  \binom{y_{n+1}^{\left(  2\right)  }}{y_{n-1}^{\left(  1\right)  }}\\
& =\binom{\left(  A_{n}-I\right)  y_{n+1}^{\left(  2\right)  }+\left(
I+B_{n}\right)  y_{n}^{\left(  2\right)  }+P_{n}y_{n}^{\left(  1\right)  }%
}{\left(  A_{n-1}-I\right)  y_{n}^{\left(  1\right)  }+\left(  I+B_{n}\right)
y_{n}^{\left(  1\right)  }+Q_{n}y_{n}^{\left(  2\right)  }}.
\end{align*}
It is clear that $L_{0}=L_{0}^{\ast}$ and $L=L_{0}+J$. Moreover, we can easily
prove that $$\sigma\left(L_{0}\right)  =\sigma_{c}\left(  L_{0}\right)  =\left[
-2,2\right],$$ and using \eqref{E4}, we get that the operator $J$ is compact in $\ell_{2}(\N,\C^{2m})$ \protect{\cite{Lusternik+Sobolev}}. By the Weyl theorem \protect{\cite[p.\ 13]{Glazman}} of a compact perturbation, we obtain
\begin{equation*}
\sigma_{c}(L)=\sigma_{c}(L_{0})=[-2,2].
\end{equation*}
This completes the proof.
\end{proof}
Since the operator $L$ is self-adjoint, the eigenvalues of $L$ is real. From
the definition of the eigenvalues, we can write
\begin{equation}\label{E13}
\sigma_{d}\left(  L\right)  =\left\{  \lambda\in \R:\lambda=-iz-\left(  iz\right)
^{-1},\text{ }iz\in\left(  -1,0\right)  \cup\left(  0,1\right)  ,\text{ }\det
F_{0}\left(  z\right)  =0\right\},
\end{equation}
where $\sigma_{d}\left(L\right)$ denotes the set of all eigenvalues of $L$.
\begin{definition}
The multiplicity of a zero of the function $\det F_{0}(z)$ is called the multiplicity of the corresponding eigenvalue of $L$.
\end{definition}
\begin{theorem}
Assume \eqref{E4}. Then the operator $L$ has a finite number of simple real eigenvalues.
\end{theorem}
\begin{proof}
To prove the theorem, we have to show that the function $\det F_{0}\left(
z\right)$ has a finite number of simple zeros. Let $z_{0}$ be one of the
zeros of $\det F_{0}\left(z\right)$. Hence $\det F_{0}\left(z_{0}\right)=0$, there is a non-zero vector $u$ such that $F_{0}\left(z_{0}\right)u=0$ \protect{\cite{Agranovich+Marchenko}}. As we know, $\binom{F_{n}\left(z\right)  }{^{G_{n}\left(z\right)}}$ is the Jost solution of \eqref{E2} for $\lambda=-iz-\left(iz\right)^{-1}$, i.e.,
\begin{equation}\label{E14}
\left\{
\begin{array}
[c]{c}%
A_{n}G_{n+1}\left(  z\right)  +B_{n}G_{n}\left(  z\right)  +P_{n}F_{n}\left(
z\right)  =\left[  -iz-\left(  iz\right)  ^{-1}\right]  F_{n}\left(  z\right)
\\
A_{n-1}F_{n-1}\left(  z\right)  +B_{n}F_{n}\left(  z\right)  +Q_{n}%
G_{n}\left(  z\right)  =\left[  -iz-\left(  iz\right)  ^{-1}\right]
G_{n}\left(  z\right).
\end{array}
\right.
\end{equation}
Differentiating \eqref{E14}  with respect to $z$, we have
\begin{equation}\label{E15}
\begin{array}
[c]{c}%
A_{n}\ddz G_{n+1}\left(z\right)  +B_{n}\ddz G_{n}\left(
z\right)+P_{n}\ddz F_{n}\left(z\right) \\
=\left[-iz-\left(iz\right)^{-1}\right]\ddz F_{n}\left(
z\right)-i\left(1-z^{-2}\right)F_{n}\left(z\right) \\
\\
A_{n-1}\ddz F_{n-1}\left(  z\right)  +B_{n}\ddz F_{n}\left(
z\right)+Q_{n}\ddz G_{n}\left(z\right) \\
=\left[-iz-\left(iz\right)^{-1}\right]\ddz G_{n}\left(
z\right)-i\left(1-z^{-2}\right)G_{n}\left(z\right).
\end{array}
\end{equation}
Using \eqref{E14} and \eqref{E15}, we obtain
\begin{align}\label{E16}
& \left(\ddz F_{n}\left(z\right)\right)^{\ast}A_{n}G_{n+1}\left(z\right)+\left(\ddz F_{n}\left(z\right)\right)^{\ast}B_{n}G_{n}\left(z\right)\\
& -\left(\ddz G_{n+1}\left(z\right)  \right)^{\ast}A_{n}
F_{n}\left(z\right)-\left(\ddz G_{n}\left(z\right)\right)^{\ast}B_{n}F_{n}\left(z\right)\nonumber\\
&=\left[-iz-\left(iz\right)^{-1}\right]\left(\ddz F_{n}\left( z\right) \right)^{\ast}F_{n}\left(z\right)\nonumber\\
&-\overline{\left[-iz-\left(iz\right)^{-1}\right]}\left(\ddz F_{n}\left(z\right)\right)^{\ast}F_{n}\left(z\right)+\overline{i\left(1-z^{-2}\right)  }F_{n}^{\ast}\left(z\right)
F_{n}\left(z\right)\nonumber
\end{align}
and
\begin{align}\label{E17}
& \left(\ddz G_{n}\left(z\right)\right)^{\ast}A_{n-1}
F_{n-1}\left(z\right)+\left(\ddz G_{n}\left(z\right)\right)^{\ast}B_{n}F_{n}\left(z\right)\\
& -\left(\ddz F_{n-1}\left(z\right)\right)^{\ast}A_{n-1}G_{n}\left(z\right)-\left(\ddz F_{n}\left(z\right)\right)^{\ast}B_{n}G_{n}\left(z\right)\nonumber\\
&=\left[-iz-\left(iz\right)^{-1}\right]\left(\ddz G_{n}\left(z\right)\right)^{\ast}G_{n}\left(z\right)\nonumber\\
&-\overline{\left[-iz-\left(iz\right)^{-1}\right]}\left(\ddz G_{n}\left(z\right)\right)^{\ast}G_{n}\left(z\right)
+\overline{i\left(1-z^{-2}\right)}G_{n}^{\ast}\left(z\right)
G_{n}\left(z\right).\nonumber
\end{align}
From \eqref{E16} and \eqref{E17}, we get
\begin{align}\label{E18}
&\left(\ddz G_{1}\left(z\right)\right)^{\ast}A_{0}F_{0}\left(z\right)-\left(\ddz F_{0}\left(z\right)\right)^{\ast}A_{0}G_{1}\left(z\right)\\
&=\left[-iz-\left(iz\right)^{-1}\right]
{\displaystyle\sum\limits_{n=1}^{\infty}}
\left[\left(\ddz F_{n}\left(z\right)\right)^{\ast}
F_{n}\left(z\right)+\left(\ddz G_{n}\left(z\right)\right)^{\ast}G_{n}\left(z\right)\right]\nonumber\\
&-\overline{\left(iz-\left(iz\right)^{-1}\right)}
{\displaystyle\sum\limits_{n=1}^{\infty}}
\left[\left(\ddz F_{n}\left(z\right)\right)^{\ast}
F_{n}\left(z\right)+\left(\ddz G_{n}\left(z\right)\right)^{\ast}G_{n}\left(z\right)\right]\nonumber\\
&+i\overline{\left(1-z^{-2}\right)}
{\displaystyle\sum\limits_{n=1}^{\infty}}
\left[F_{n}^{\ast}\left(z\right)F_{n}\left(z\right)+G_{n}^{\ast
}\left(z\right)G_{n}\left(z\right)\right].\nonumber
\end{align}
If we write \eqref{E18} for $z=z_{0}$, we obtain
\begin{align}\label{E19}
&\left(\ddz G_{1}\left(z_{0}\right)\right)^{\ast}A_{0}
F_{0}\left(z_{0}\right)-\left(\ddz F_{0}\left(z_{0}\right)\right)^{\ast}A_{0}G_{1}\left(z_{0}\right)\\
&=-i\left(1-\overline{z_{0}^{-2}}\right)
{\displaystyle\sum\limits_{n=1}^{\infty}}
\left[F_{n}^{\ast}\left(z_{0}\right)F_{n}\left(z_{0}\right)
+G_{n}^{\ast}\left(z_{0}\right)G_{n}\left(z_{0}\right)\right]\nonumber
\end{align}
using $iz_{0}\in(-1,0)\cup(0,1)$.
Then if we multiply \eqref{E19}  with the vector $u$ on the right side $(u\in\ell_{2}(\N,\C^{2m}), u\neq0)$,  we get
$$\left\langle A_{0}G_{1}\left(z_{0}\right)u,\ddz F_{0}\left(z_{0}\right)u\right\rangle =\left(i-\frac{i}{\left(iz_{0}\right)^{2}}\right)\left\{
{\displaystyle\sum\limits_{n=1}^{\infty}}
\left\Vert F_{n}\left(z_{0}\right)u\right\Vert^{2}+
{\displaystyle\sum\limits_{n=1}^{\infty}}
\left\Vert G_{n}\left(z_{0}\right)u\right\Vert^{2}\right\}.$$
Since $iz_{0}\neq0$ and $iz_{0}\neq1$, we can write $i-\dfrac{i}{\left(iz_{0}\right)^{2}}\neq0$. Also we can write $\left\Vert F_{n}\left(  z_{0}\right)u\right\Vert \neq0$
and $\left\| G_{n}\left(  z_{0}\right)u\right\| \neq0$ for all $n\in\N$, so
$$\left\langle A_{0}G_{1}\left(z_{0}\right)  u,\frac{d}{dz}F_{0}\left(
z_{0}\right)  u\right\rangle \neq0.$$
This shows that $\ddz F_{0}\left(z_{0}\right)u\neq0,$ that is, all
zeros of $\det F_{0}\left(z\right)$ are simple. To complete the proof of theorem, we have to show that the function $\det
F_{0}\left(z\right)$ has a finite number of zeros. Let we take the
function
\[
M\left(  z\right)  =z^{-1}\left(  T_{0}^{11}\right)  ^{-1}F_{0}\left(
z\right)  =I+A\left(  z\right),
\]
where $A\left(  z\right)  =%
{\displaystyle\sum\limits_{m=1}^{\infty}}
K_{0m}^{11}z^{2m}-i%
{\displaystyle\sum\limits_{m=1}^{\infty}}
K_{0m}^{12}z^{2m-1}$. Since $A\left(  z\right)$ is matrix-valued analytic function
on $D$, the function $M$ has inverse on the boundary of $D$ \protect{\cite[Theorem 5.1]{Gohberg+Krein}}, i.e., the limit points of
the set of zeros of
\begin{equation}\label{E20}
\det F_{0}\left(z\right)=0
\end{equation}
is empty. Therefore, the set of zeros of \eqref{E20} in $D$ is
finite, i.e., the operator $L$ has a finite number of eigenvalues.
\end{proof}
bibliographystyle{plain}

\bibliographystyle{plain}
\bibliography{tttbib}
\end{document}